\documentclass[11pt]{article}
\usepackage{amsmath,amsthm,amssymb}
\usepackage[usenames,dvipsnames]{xcolor}
\usepackage{enumerate}
\usepackage{graphicx}
\usepackage{cite}
\usepackage{comment}
\usepackage{oands}
\usepackage{tikz}
\usepackage{changepage}
\usepackage{bbm}
\usepackage{mathtools}
\usepackage[margin=1in]{geometry}
\usepackage[pagewise,mathlines]{lineno}
\usepackage{appendix}
\usepackage{stmaryrd} 
\usepackage{multicol}
\usepackage{microtype}
\usepackage[colorinlistoftodos]{todonotes}
\usepackage{enumitem}

\usepackage[pdftitle={Bounds for distances and geodesic dimension in Liouville first passage percolation},
  pdfauthor={Ewain Gwynne and Joshua Pfeffer},
colorlinks=true,linkcolor=NavyBlue,urlcolor=RoyalBlue,citecolor=PineGreen,bookmarks=true,bookmarksopen=true,bookmarksopenlevel=2,unicode=true,linktocpage]{hyperref}

\theoremstyle{plain}
\newtheorem{thm}{Theorem}[section]
\newtheorem{cor}[thm]{Corollary}
\newtheorem{lem}[thm]{Lemma}

\def\@rst #1 #2other{#1}
\newcommand\MR[1]{\relax\ifhmode\unskip\spacefactor3000 \space\fi
  \MRhref{\expandafter\@rst #1 other}{#1}}
\newcommand{\MRhref}[2]{\href{http://www.ams.org/mathscinet-getitem?mr=#1}{MR#2}}

\theoremstyle{definition}

\newtheorem{remark}[thm]{Remark}

\numberwithin{equation}{section}

\newcommand{\dsb}{\begin{adjustwidth}{2.5em}{0pt}
\begin{footnotesize}}
\newcommand{\dse}{\end{footnotesize}
\end{adjustwidth}}

\newcommand{\ssb}{\begin{adjustwidth}{2.5em}{0pt}}
\newcommand{\sse}{\end{adjustwidth}}

\newcommand{\aryb}{\begin{eqnarray*}}
\newcommand{\arye}{\end{eqnarray*}}
\def\alb#1\ale{\begin{align*}#1\end{align*}}
\def\allb#1\alle{\begin{align}#1\end{align}}
\newcommand{\eqb}{\begin{equation}}
\newcommand{\eqe}{\end{equation}}
\newcommand{\eqbn}{\begin{equation*}}
\newcommand{\eqen}{\end{equation*}}

\newcommand{\BB}{\mathbbm}
\newcommand{\ol}{\overline}
\newcommand{\ul}{\underline}
\newcommand{\op}{\operatorname}

\newcommand{\ep}{\varepsilon}
\newcommand{\rta}{\rightarrow}

\newcommand{\wt}{\widetilde}
 
\newcommand{\mcl}{\mathcal}

\newcommand{\bdy}{\partial}
\newcommand{\rng}{\mathring}

\newcommand{\cc}{\mathbf{c}}

\let\originalleft\left
\let\originalright\right
\renewcommand{\left}{\mathopen{}\mathclose\bgroup\originalleft}
\renewcommand{\right}{\aftergroup\egroup\originalright}

\title{Bounds for distances and geodesic dimension in Liouville first passage percolation}

\date{ }
\author{
\begin{tabular}{c} Ewain Gwynne\\[-5pt]\small Cambridge \end{tabular}  
\begin{tabular}{c} Joshua Pfeffer\\[-5pt]\small MIT \end{tabular} 
}

\begin{document}

\maketitle

\begin{abstract}
For $\xi \geq 0$, Liouville first passage percolation (LFPP) is the random metric on $\varepsilon \mathbb Z^2$ obtained by weighting each vertex by $\varepsilon  e^{\xi h_\varepsilon(z)}$, where $h_\varepsilon(z)$ is the average of the whole-plane Gaussian free field $h$ over the circle $\partial B_\varepsilon(z)$. 
Ding and Gwynne (2018) showed that for $\gamma \in (0,2)$, LFPP with parameter $\xi = \gamma/d_\gamma$ is related to $\gamma$-Liouville quantum gravity (LQG), where $d_\gamma$ is the $\gamma$-LQG dimension exponent. 
For $\xi  > 2/d_2$, LFPP is instead expected to be related to LQG with central charge greater than 1. 

We prove several estimates for LFPP distances for general $\xi\geq 0$. For $\xi\leq 2/d_2$, this leads to new bounds for $d_\gamma$ which improve on the best previously known upper (resp.\ lower) bounds for $d_\gamma$ in the case when $\gamma  > \sqrt{8/3}$ (resp.\ $\gamma \in (0.4981, \sqrt{8/3})$). 
These bounds are consistent with the Watabiki (1993) prediction for $d_\gamma$. 
However, for $\xi > 1/\sqrt 3$ (or equivalently for LQG with central charge larger than 17) 
our bounds are inconsistent with the analytic continuation of Watabiki's prediction to the $\xi  >2/d_2$ regime. 
We also obtain an upper bound for the Euclidean dimension of LFPP geodesics. 
\end{abstract}


\section{Introduction}
\label{sec-intro}

Let $\gamma \in (0,2]$, let $U\subset\BB C$ be an open set, and let $h$ be some variant of the Gaussian free field (GFF) on $U$. The \emph{$\gamma$-Liouville quantum gravity (LQG)} surface associated with $(U,h)$ is, heuristically speaking, the random two-dimensional Riemannian manifold parametrized by $U$ with Riemannian metric tensor $e^{\gamma h} \, (dx^2 + dy^2)$, where $dx^2 + dy^2$ is the Euclidean metric tensor. LQG surfaces arise as the scaling limits of various discrete random geometries, such as random planar maps and Liouville first passage percolation, which we discuss just below.

The above definition of an LQG surface does not make literal sense since $h$ is a distribution, not a function. 
Nevertheless, it is possible to make sense of the volume form associated with an LQG surface as a limit of regularized versions of $e^{\gamma h} \,dz$, where $dz$ denotes Lebesgue measure; see~\cite{kahane,shef-kpz,rhodes-vargas-review}. 
One can also make sense of LQG as a random metric space. This was first done in the special case when $\gamma=\sqrt{8/3}$ by Miller and Sheffield~\cite{lqg-tbm1,lqg-tbm2,lqg-tbm3}, in which case the resulting metric space is isometric to the so-called \emph{Brownian map}~\cite{legall-uniqueness,miermont-brownian-map}.
Very recently, Gwynne and Miller~\cite{gm-uniqueness} constructed the $\gamma$-LQG metric for all $\gamma\in(0,2)$, building on the works~\cite{dddf-lfpp,lqg-metric-estimates,local-metrics,gm-confluence}. 

Several important properties of the $\gamma$-LQG are not yet fully understood. 
For example, for $\gamma \not=  \sqrt{8/3} $, the Hausdorff dimension of the $\gamma$-LQG metric is unknown (the dimension is 4 for $\gamma=\sqrt{8/3}$).
However, recent progress on related problems has been made in~\cite{ding-goswami-watabiki,ding-dunlap-lqg-fpp,ding-zhang-geodesic-dim,dzz-heat-kernel,ghs-dist-exponent,ghs-map-dist,dg-lqg-dim,df-lqg-metric,ding-dunlap-lgd}. 
Particularly relevant to us are the articles~\cite{ghs-map-dist,dzz-heat-kernel,dg-lqg-dim} which establish for each $\gamma \in (0,2)$ the existence of an exponent $d_\gamma > 2$ which arises in a variety of different approximations of LQG distances and which is equal to the Hausdorff dimension of the LQG metric~\cite{gp-kpz}.
See~\eqref{eqn-lambda-gamma} below for the appearance of $d_\gamma$ in our paper.
We define $d_2 := \lim_{\gamma\rta 2^-} d_\gamma$, which exists since $\gamma\mapsto d_\gamma$ is increasing~\cite{dg-lqg-dim}. 
It is known that $d_{\sqrt{8/3}} = 4$~\cite[Theorem 1.2]{dg-lqg-dim}, but for other $\gamma$ the value of $d_\gamma$ is unknown. The best-known physics prediction, due to Watabiki~\cite{watabiki-lqg}, is 
\eqb \label{eqn-watabiki} 
d_\gamma^{\op{Wat}} = 1 + \frac{\gamma^2}{4} + \frac14\sqrt{(4+\gamma^2)^2 + 16\gamma^2} .
\eqe
But, it was proven by Ding and Goswami~\cite{ding-goswami-watabiki} that this prediction is false for small values of $\gamma$.

One of the most natural ways to study $\gamma$-LQG distances is to consider the random metric obtained by exponentiating a continuous approximation of the GFF  (as is done in several of the above-cited works). 
Such approximate metrics are referred to as \emph{Liouville first passage percolation (LFPP)}. 
In this article, we will prove several estimates for LFPP distances which in particular lead to new bounds for $d_\gamma$ for general $\gamma \in (0,2]$, improving on the previous best known upper (resp.\ lower) bound from~\cite{dg-lqg-dim} in the case when $\gamma > \sqrt{8/3}$ (resp.\ $\gamma \in (0.4981, \sqrt{8/3})$) (Corollary~\ref{cor-d-bound}). We also establish an upper bound for the Euclidean dimension of LFPP geodesics.

Our bounds are valid not only for discretizations of $\gamma$-LQG with $\gamma \in (0,2]$ but also for discreteizations of a certain extension of LQG beyond this phase: LQG with central charge in $(1,25)$, which corresponds to $\xi > 2/d_2$ in the model which we define just below. 
In this extended regime, our bounds are inconsistent with the analytic continuation of Watabiki's prediction for a range of parameter values; see Corollary~\ref{cor-watabiki-contradict}.
\medskip

\noindent\textbf{Acknowledgments.} We thank an anonymous referee for helpful comments on an earlier version of this article. We thank Scott Sheffield for helpful discussions. Part of the project was carried out during J.P.'s visit to the Isaac Newton Institute in Cambridge, UK in Summer 2018. We thank the institute for its hospitality.  J.P.\ was partially supported by the National Science Foundation Graduate Research Fellowship under Grant No. 1122374.
 
\subsection*{Definition of the model}

Let $h$ be a whole-plane GFF, normalized so that its circle average over $\bdy\BB D$ is zero. For $\ep > 0$ and $z\in\BB C$, we write $h_\ep(z)$ for the average of $h$ over the circle of radius $\ep$ centered at $z$ (see~\cite[Section 3.1]{shef-kpz} for more on the circle average process). 
We write $\BB S  := [0,1]^2$. 
For $\ep > 0$, we define $\BB S^\ep := (\ep\BB Z^2) \cap \BB S$ and we equip $\BB S^\ep$ with its standard nearest-neighbor graph structure. 

For $\ep , \xi \geq 0$ and a lattice path $P :\{0,1,\ldots,N\}   \rta  \BB S^\ep$ for some $N \in\BB N$, we define the \emph{$\ep$-LFPP length} of $P$, with parameter $\xi$, by
\eqb
L_h^{\xi,\ep}(P) := \sum_{j=0}^N \ep e^{\xi h_\ep(P(j))} .
\eqe
The reason for the factor of $\ep$ is that edges of $\BB Z^2$ have side length $\ep$, so this factor makes it so that $L_h^{\xi,\ep}$ approximates the integral of $e^{\xi h_\ep}$ along a linearly interpolated version of $P$.  
For $z,w\in \BB S^\ep$, we define the \emph{$\xi$-LFPP distance} by\footnote{\label{footnote-lfpp-variants} There are several other natural variants of LFPP besides the one we consider here. 
For example, we can replace the circle average process by the white noise approximation or by the discrete GFF. 
We can also define distances by integrating along continuous paths rather than by summing along paths in $\ep\BB Z^2$. 
The arguments of this paper work for any of these approximations; all we need is that the variance of the approximating field is of order $\log\ep^{-1}  + O_\ep(1)$, uniformly over $\BB S$. 
It follows from~\cite[Lemma 3.1 and Proposition 3.16]{dg-lqg-dim} that with probability tending to 1 as $\ep\rta 0$, LFPP distances are scaled by a factor of at most $\ep^{o_\ep(1)}$ if we replace the circle average process by the white noise decomposition and/or we integrate along continuous paths instead of discrete paths.
\cite[Theorem 1.4]{ang-discrete-lfpp} gives a similar statement comparing LFPP defined using the discrete GFF instead of the circle average process of the continuum GFF.
In particular, the exponents for distances in the above variants of LFPP are all the same. 
We will use this fact without comment when we cite results from~\cite{dg-lqg-dim}. }
\eqb \label{eqn-lfpp-def} 
D_h^{\xi,\ep}(z,w  ) := \inf_{P : z\rta w}  L_h^{\xi,\ep}(P)  ,
\eqe
where the infimum is over all lattice paths in $\BB S^\ep$ from $z$ to $w$.   
Let $\bdy_{\op{L}} \BB S^\ep$ (resp.\ $\bdy_{\op{R}} \BB S^\ep$) be the left (resp.\ right) boundary of $\BB S^\ep$, i.e., the set of vertices of $\BB S^\ep$ whose nearest neighbor to the left (resp.\ right) in $\ep\BB Z^2$ is not in $\BB S^\ep$.  
For $\xi \geq 0$, we define the \emph{$\xi$-LFPP distance exponent}
\eqb \label{eqn-lambda-def}
\lambda(\xi) := \sup\left\{ \alpha \in\BB R  : \liminf_{\ep\rta 0} \BB P\left[ D_h^{\xi,\ep}\left( \bdy_{\op{L}} \BB S^\ep , \bdy_{\op{R}} \BB S^\ep \right) \leq \ep^{ \alpha} \right] = 1 \right\} .
\eqe
We have the following \emph{a priori} bounds for $\lambda(\xi)$.

\begin{lem} \label{lem-lambda-basic}
$\lambda(\xi) \in [-1/2 , 1]$ for all $\xi \geq 0$. Furthermore, $\xi\mapsto \lambda(\xi)$ is 2-Lipschitz on $[0,\infty)$. 
\end{lem}
\begin{proof}
Let $\zeta >0$ be a small exponent. By~\cite[Lemma 5.3]{ghpr-central-charge}, it holds with probability tending to 1 as $\ep\rta 0$ that there is a path $P$ from $\bdy_{\op{L}} \BB S^\ep$ to $\bdy_{\op{R}} \BB S^\ep$ whose total number of vertices satisfies $\# P \leq \ep^{-3/2-\zeta}$ such that $|h_\ep(P(j))| \leq \zeta \log \ep^{-1} $ for each $j=0,\dots,\# P-1$. For this path, one has $L_h^{\xi,\ep}(P) \leq \ep^{-1/2- (\xi+1)\zeta}$, so $\lambda(\xi) \geq -1/2- (\xi+1) \zeta$. 
Furthermore, any path between the upper and lower boundaries of $\BB S^\ep$ must cross $P$, so must have $L_h^{\xi,\ep}$-length at least $\ep^{1+\xi\zeta}$. By $\pi/2$-rotational symmetric this implies $\lambda(\xi) \leq 1+\xi \zeta$. Since $\zeta>0$ is arbitrary we get $\lambda(\xi) \in [-1/2,1]$.

For the Lipschitz continuity, we observe that $\sup_{z\in\BB S} |h_\ep(z)| \leq (2+\zeta) \log\ep^{-1}$ with probability tending to 1 as $\ep\rta 0$ (see, e.g.,~\cite[Proposition 2.4]{lqg-tbm2}). On the event that this bound holds, for any path $P$ and any $0\leq \wt\xi\leq \xi$, we have $\ep^{ (2+\zeta)(\xi-\wt\xi)} \leq L_h^{\xi,\ep}(P) / L_h^{\wt\xi,\ep}(P) \leq \ep^{-(2+\zeta)(\xi-\wt\xi)}$. Therefore $|\lambda(\xi) - \lambda(\wt\xi)| \leq (2+\zeta)(\xi-\wt\xi)$. Again since $\zeta > 0$ is arbitrary, we get the 2-Lipschitz continuity. 
\end{proof}

It is shown in~\cite[Theorem 1.5]{dg-lqg-dim} (c.f.\ Footnote~\ref{footnote-lfpp-variants}) that, with $d_\gamma$ the dimension exponent for $\gamma$-LQG, it holds with probability tending to 1 as $\ep\rta 0$ that $D_h^{\gamma/d_\gamma , \ep}(\bdy_{\op{L}} \BB S^\ep , \bdy_{\op{R}}\BB S^\ep) = \ep^{1 - \gamma Q/d_\gamma + o_\ep(1)}$, where $Q = 2/\gamma + \gamma/2$. Moreover, the same is true for several other quantities related to LFPP distances, such as diameters and point-to-point distances.
In particular, 
\eqb \label{eqn-lambda-gamma}
\lambda(\gamma/d_\gamma) = 1 - \frac{\gamma}{d_\gamma} Q  = 1 - \frac{\gamma}{d_\gamma} \left(\frac{2}{\gamma} + \frac{\gamma}{2}\right) ,\quad\forall \gamma \in (0,2] .
\eqe
Note that~\eqref{eqn-lambda-gamma} for $\gamma=2$ follows from the case $\gamma <2$ and the continuity of $\lambda$. Since $d_{\sqrt{8/3}} = 4$,~\eqref{eqn-lambda-gamma} implies that $\lambda(1/\sqrt 6) =1/6 $.

As we will explain in Section~\ref{sec-cc}, we expect that LFPP with $\xi > 2/d_2$ is connected to Liouville quantum gravity with central charge in $(1,25]$ (note that $\gamma$-LQG for $\gamma\in (0,2]$ corresponds to central charge in $(-\infty,1]$).
In this regime, we do not know that $ D_h^{\xi,\ep}\left( \bdy_{\op{L}} \BB S^\ep , \bdy_{\op{R}} \BB S^\ep \right) \geq \ep^{\lambda(\xi) + o_\ep(1)}$ with high probability. 
However, we expect that this can be proven using arguments similar to those used to show the existence of an exponent for Liouville graph distance in~\cite{dzz-heat-kernel}.

By~\eqref{eqn-lambda-gamma}, Watabiki's prediction~\eqref{eqn-watabiki} is equivalent to $\lambda(\xi)  =  \xi^2 $, $\forall \xi \in [0,2/d_2]$: indeed, $d_\gamma^{\op{Wat}}$ is the positive solution to $d_\gamma^2  - \gamma Q d_\gamma =  \gamma^2  $, and dividing this by $d_\gamma^2$ gives $1-\xi Q  = \xi^2  $. 

One might guess that $\lambda(\xi)$ is an analytic function of $\xi$ at least up until the smallest $\xi > 0$ for which $\lambda(\xi) = 1$. Indeed, as explained in Section~\ref{sec-cc}, $\lambda(\xi) = 1$ corresponds to the critical point at which the ``background charge" $Q$ is zero and the ``central charge" $\cc$ is 25.
We do not know if there is a finite value of $\xi_* > 0$ for which $\lambda(\xi_*) =1$, but if such a $\xi_*$ exists we expect that LFPP is in some sense degenerate for $\xi > \xi_*$. The vast majority of exponents associated with LQG depend analytically on their parameter values in the range for which the objects are non-degenerate. 
For example, the KPZ formula extends analytically to the case when $\cc \in (1,25)$~\cite[Theorem 1.5]{ghpr-central-charge}.\footnote{If one assumes that a metric on LQG with $\xi  > 2/d_2$ and $Q\in (0,2)$ exists and satisfies certain axioms (which should be satisfied if LFPP for $\xi > 2/d_2$ has a scaling limit), then one can show that the formulas for exponents / dimensions from~\cite{gp-kpz,lqg-metric-estimates} extend analytically to this regime, provided they still give finite positive answers. Examples of such formulas include the optimal H\"older exponent of the Euclidean metric w.r.t.\ the LQG metric~\cite[Theorem 1.7]{lqg-metric-estimates} and the LQG dimension of the $\alpha$-thick points of the field~\cite[Theorem 1.5]{gp-kpz}. The proofs of these formulas should be the same as the proofs in~\cite{gp-kpz,lqg-metric-estimates}, but such proofs have not been written down.}
We emphasize, though, that the analyticity of $\lambda(\xi)$ for $\xi < \xi_*$ is only a guess, and we would not be very surprised if this guess turns out to be false. 

If $\lambda(\xi)$ were analytic for $\xi < \xi_*$, then we would have the following extension of Watabiki's prediction from the case when $\xi \in (0,2/d_2]$ to the case of general $\xi >0$:
\eqb \label{eqn-watabiki-extend}
\lambda^{\op{Wat}}(\xi)  =  \min\{\xi^2 ,1\}  ,\quad\forall \xi  \geq 0 .
\eqe
We will show in Corollary~\ref{cor-watabiki-contradict} that~\eqref{eqn-watabiki-extend} is false for a specific subset of $[2/d_2,\infty)$.

\section{Main results}
\label{sec-results}

The starting point of our main results is the following comparison of LFPP lengths of a path for different values of $\xi$, which will be proven (via a one-page argument) in Section~\ref{sec-proof}.

\begin{thm} \label{thm-length-compare}
Let $0 \leq \wt\xi \leq \xi  $ and fix a small parameter $\zeta > 0$.  
With probability tending to 1 as $\ep\rta 0$, each simple path $P$ in $\BB S^\ep$ with $D_h^{\xi,\ep}$-length $L_h^{\xi,\ep}(P) \leq \ep^{\lambda(\xi) -\zeta}$ satisfies
\eqb \label{eqn-length-compare}
L_h^{\wt\xi,\ep}(P) \leq  \ep^{\lambda(\xi)   -  (\xi-\wt\xi) \left( \sqrt{2 + 2 \lambda(\xi) + \xi^2}  -\xi  \right) - \zeta } .
\eqe 
\end{thm}


\begin{cor}[Upper differential inequality for $\lambda$] \label{cor-lambda-mono}
If $ 0 \leq \wt\xi < \xi$, then
\eqb \label{eqn-lambda-mono}
\frac{\lambda(\xi) -  \lambda(\wt\xi) }{\xi - \wt \xi} \leq   \sqrt{2 + 2 \lambda(\xi) + \xi^2}  -\xi    .
\eqe 
In particular, for Lebesgue-a.e.\ $\xi  \geq 0$, 
\eqb \label{eqn-lambda-deriv}
\lambda'(\xi) \leq  \sqrt{2 + 2 \lambda(\xi) + \xi^2}  -\xi  .
\eqe 
\end{cor}
\begin{proof}
The relation~\eqref{eqn-lambda-mono} is immediate from the definition~\eqref{eqn-lambda-def} of $\lambda$ and Theorem~\ref{thm-length-compare} applied to a path from $\bdy_{\op{L}} \BB S^\ep$ to $\bdy_{\op{R}} \BB S^\ep$ with minimal $D_h^{\xi,\ep}$-length. 
This relation implies~\eqref{eqn-lambda-deriv} since $\xi\mapsto \lambda(\xi)$ is Lipschitz, hence differentiable a.e. 
\end{proof}

Corollary~\ref{cor-lambda-mono} complements the monotonicity relation from~\cite[Lemma 2.5]{dg-lqg-dim} which shows that for $0 \leq \wt\xi \leq \xi$, 
\eqb \label{eqn-dg-mono}
\lambda(\wt\xi) + \frac{\wt\xi^2}{2} \leq \lambda(\xi) + \frac{\xi^2}{2} \quad \text{and hence} \quad \lambda'(\xi) \geq - \xi .
\eqe
By combining these two differential inequalities and the fact that $\lambda(0) = 0$ and $\lambda(1/\sqrt 6) = 1/6$ (see the discussion just below~\eqref{eqn-lambda-gamma}), we get the following theorem.

\begin{thm}[Bounds for $\lambda(\xi)$] \label{thm-lambda-bound}
For $\xi  \geq 0$, 
\eqb \label{eqn-lambda-bound}
\ul\lambda(\xi) \leq \lambda(\xi) \leq \ol\lambda(\xi)
\eqe
where
\eqb \label{eqn-lambda-lower}
\ul\lambda(\xi) := 
\begin{dcases}
\max\left\{ \left(\sqrt{\frac{5}{2}} - \frac{1}{\sqrt 6} \right) \xi  - \frac{\sqrt{15}-2}{6}   ,   - \frac{\xi^2}{2}     \right\} ,\quad &\xi \leq \frac{1}{\sqrt 6}  \\
\max\left\{ \frac{1}{4} - \frac{\xi^2}{2} ,  - \frac12 \right\} \quad &\xi \geq \frac{1}{\sqrt 6} 
\end{dcases}
\eqe 
and 
\eqb \label{eqn-lambda-upper}
\ol\lambda(\xi) := 
\begin{dcases}
\min\left\{    \frac{1}{4} - \frac{\xi^2}{2} ,  \sqrt{2}\xi    \right\} ,\quad &\xi \leq \frac{1}{\sqrt 6}  \\
\min\left\{  \left(\sqrt{\frac{5}{2}} - \frac{1}{\sqrt 6} \right) \xi  - \frac{\sqrt{15} -2}{6}    ,   1     \right\} ,\quad &\xi \geq \frac{1}{\sqrt 6} .
\end{dcases}
\eqe  
\end{thm}
\begin{proof}
Recall from Lemma~\ref{lem-lambda-basic} that $\lambda(\xi) \in [-1/2, 1]$ for all $\xi\geq 0$. 
Since $\lambda(0) = 0$, by setting $\wt\xi = 0$ in~\eqref{eqn-lambda-mono} and~\eqref{eqn-dg-mono} and solving for $\lambda(\xi)$, we get
\eqb \label{eqn-around0}
- \frac{\xi^2}{2}  \leq  \lambda(\xi)  \leq \sqrt{2}\xi   ,\quad\forall \xi \geq 0. 
\eqe
Since $\lambda(1/\sqrt 6) = 1/6$, by setting $\wt\xi = 1/6$ in~\eqref{eqn-lambda-mono} and~\eqref{eqn-dg-mono} and solving for $\lambda(\xi)$, we get  
\eqb \label{eqn-around-4}
 \frac{1}{4} - \frac{\xi^2}{2} \leq  \lambda(\xi) \leq  \left(\sqrt{\frac{5}{2}} - \frac{1}{\sqrt 6} \right) \xi  - \frac{\sqrt{15}-2}{6} ,\quad\forall\xi \geq 1/\sqrt 6 .
\eqe
By instead setting $\xi = 1/6$ and solving for $\lambda(\wt\xi)$, we get~\eqref{eqn-around-4} with the inequality signs flipped for $\xi \leq 1/\sqrt 6$. 
Combining these inequalities gives~\eqref{eqn-lambda-bound}. 
\end{proof}

See Figure~\ref{fig-lambda-bound}, left, for a plot of the bounds~\eqref{eqn-lambda-bound}. 
At the time this paper was written, the bounds for $\lambda(\xi)$ from Theorem~\ref{thm-lambda-bound} were the best known except when $\xi$ is very small (non-explicit), in which case~\cite{ding-goswami-watabiki} gives $\lambda(\xi) \geq  c \xi^{4/3}/\log(1/\xi) $ for a non-explicit universal constant $c>0$. However, very recently improved lower bounds have been obtained in some cases; see Remark~\ref{remark-ang}. From Corollary~\ref{cor-lambda-mono} and Theorem~\ref{thm-lambda-bound}, we get the following.
 
\begin{cor} \label{cor-watabiki-contradict}
The extended Watabiki prediction~\eqref{eqn-watabiki-extend} is false on a dense subset of $\{ \xi \geq 1/\sqrt 3 : \lambda(\xi) < 1 \}$ and for all $\xi\in \left(\sqrt{\frac52}-\sqrt{\frac23} , \frac{4 +\sqrt{15}}{ \sqrt 2 ( 3 \sqrt{5} - \sqrt 3)} \right)\approx \left(0.7646 ,  1.1187  \right)$. 
\end{cor}
\begin{proof}
If $\lambda(\xi) = \xi^2$ holds on a neighborhood of $\xi$ and $\xi > 1/\sqrt 3$, then $\lambda'(\xi) = 2\xi < \sqrt{2+2\xi^2 + \xi^2} -\xi$, contrary to~\eqref{eqn-lambda-deriv}. The second statement follows since the upper bound in~\eqref{eqn-lambda-upper} is strictly less than $\xi^2$ for $\xi > \sqrt{\frac52}-\sqrt{\frac23}$ and strictly less than 1 for $\xi < \frac{4 +\sqrt{15}}{ \sqrt 2 ( 3 \sqrt{5} - \sqrt 3)} $. 
\end{proof}

As explained in Section~\ref{sec-cc}, the conditions in Corollary~\ref{cor-watabiki-contradict} correspond to central charge in $(17,25)$ and in $(21.741\dots,25)$, respectively. 
The combination of Corollary~\ref{cor-watabiki-contradict} and~\cite{ding-goswami-watabiki} shows that if Watabiki's prediction is true for a non-trivial interval of $\xi$-values, then $\lambda(\xi)$ must fail to be analytic at at least two different values of $\xi$ with $\lambda(\xi )  <1$. 
Since the existence of two non-analytic points would be more surprising than the existence of just one such point, this provides further evidence against the statement that Watabiki's prediction is true for a non-trivial interval of $\xi$-values. 

In contrast, all known bounds for $\lambda(\xi)$ (including those of Corollary~\ref{cor-lambda-mono}, Theorem~\ref{thm-lambda-bound}, and~\cite{ding-goswami-watabiki}) \emph{are} consistent with the alternative guess for the dimension of $\gamma$-LQG from~\cite[Equation (1.16)]{dg-lqg-dim}, namely $d_\gamma^{\op{DG}} = 2 + \frac{\gamma^2}{2} + \frac{\gamma}{\sqrt 6}$,
which is equivalent to $\lambda(\xi) = \xi/\sqrt 6$ for $\xi \in [0,2/d_2]$ and corresponds to the extended guess 
\eqb \label{eqn-dg-guess}
\lambda^{\op{DG}}(\xi) = \min\{\xi/\sqrt 6 , 1\} ,\quad \forall \xi \geq 0.
\eqe
Note that if $\lambda(\xi) = \lambda^{\op{DG}}(\xi)$, then the Ding-Goswami bound $\lambda(\xi) \geq  c \xi^{4/3}/\log(1/\xi)$ for small $\xi$ would be far from optimal. We emphasize, though, that there is currently no theoretical justification for the above alternative guess, even at a heuristic level.

\begin{figure}[t!]
 \begin{center}
\includegraphics[scale=.55]{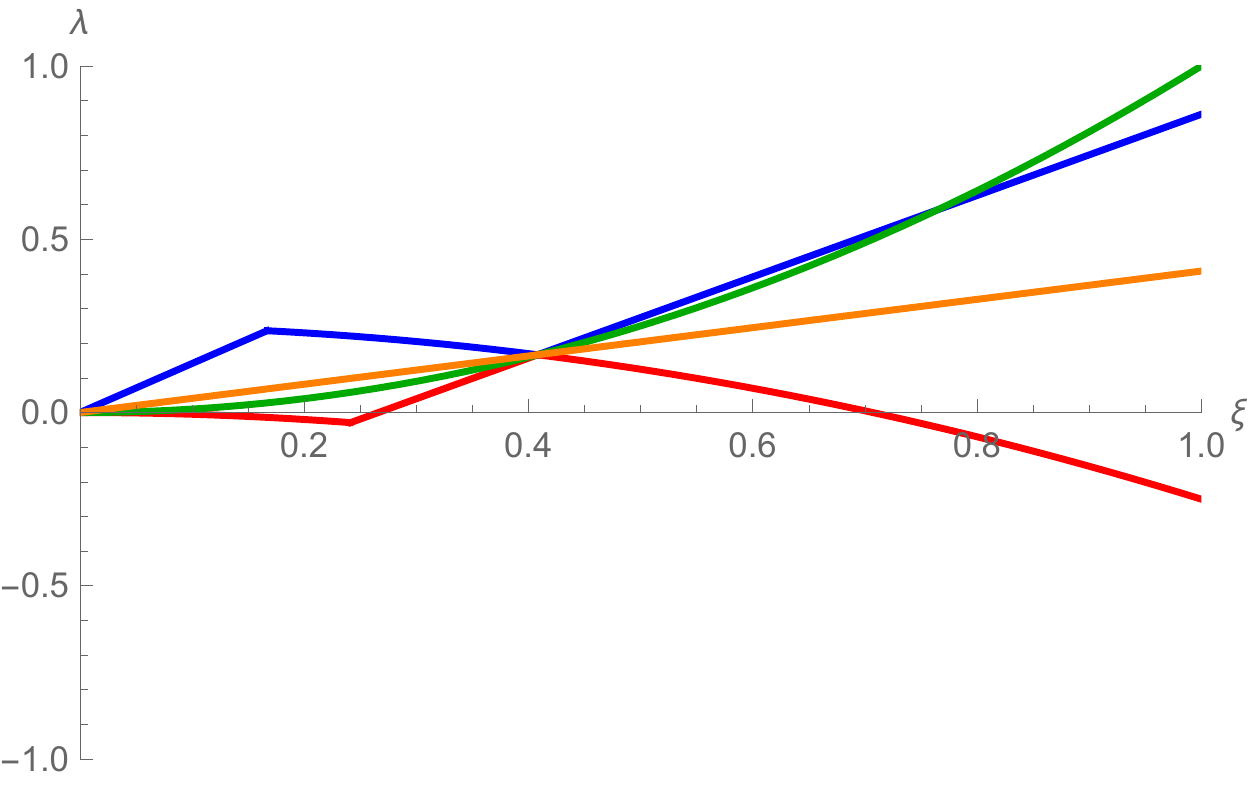} \hspace{15pt} \includegraphics[scale=.55]{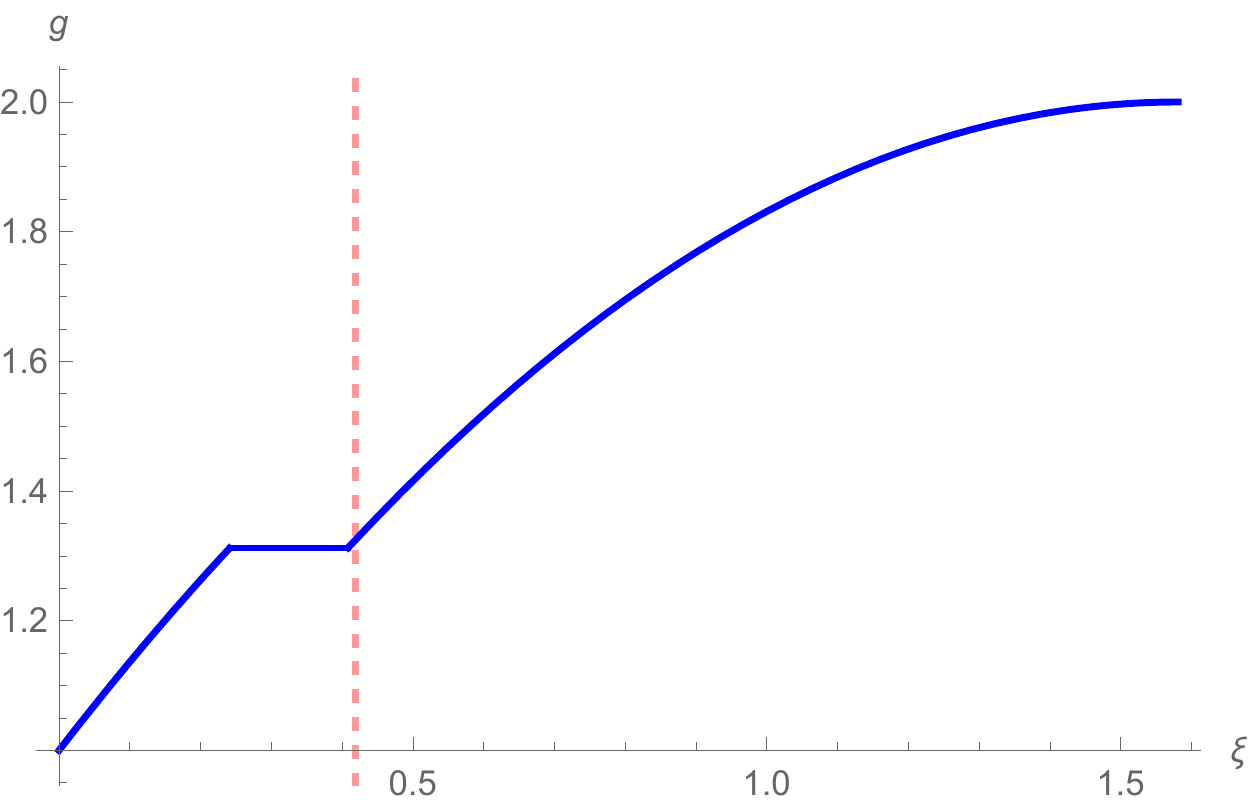}
\vspace{-0.01\textheight}
\caption{ \textbf{Left.} Graph of the lower bound $\ul \lambda(\xi)$ (red) and the upper bound $\ol \lambda(\xi)$ (blue) from Theorem~\ref{thm-lambda-bound} together with the extended Watabiki prediction $\lambda(\xi) = \max\{\xi^2,1\}$ (green) and the alternative guess $\lambda(\xi) = \max\{\xi/\sqrt 6 , 1\}$ (orange) for $\xi \in [0,1]$. 
\textbf{Right.} Graph of the upper bound for the geodesic dimension $g(\xi)$, obtained from plugging in the lower bound for $\lambda(\xi)$ from Theorem~\ref{thm-lambda-bound} into Corollary~\ref{cor-geodesic-dim-upper}, on the interval $[0,\sqrt{5/2}]$ (the range on which they are non-trivial). The upper bound is constant at $\frac16(4+\sqrt{15})\approx 1.31216$ for $\xi \in \left[ 0.241\dots , 1/\sqrt 6 \right]$ but we expect that $g(\xi)$ is strictly increasing, at least on $[0,2/d_2]$. The dashed red line is at our upper bound $2 - \sqrt{5/2} \approx 0.4189$ for $2/d_2$. The vertical coordinate where it crosses the graph is $2\sqrt{10} - 5\approx 1.3246$, which is (at least heuristically) an upper bound for the largest possible Euclidean dimension of a $\gamma$-LQG geodesic for $\gamma \in (0,2]$. 
}\label{fig-lambda-bound}
\end{center}
\vspace{-1em}
\end{figure} 

Using~\eqref{eqn-lambda-gamma}, we can translate Theorem~\ref{thm-lambda-bound} for $\xi\in (0,2/d_2]$ into bounds for the $\gamma$-LQG dimension.

\begin{cor}[Bounds for $d_\gamma$] \label{cor-d-bound}
For $\gamma \in (0,2)$, one has
\eqb 
\ul d_\gamma \leq d_\gamma \leq \ol d_\gamma
\eqe
for
\eqb \label{eqn-d-lower}
\ul d_\gamma := 
\begin{dcases}
\max\left\{  \frac{ 12 - \sqrt 6 \gamma + 3 \sqrt{10} \gamma + 3 \gamma^2}{4 + \sqrt{15}}    ,  \frac{2\gamma^2}{4+\gamma^2-\sqrt{16 +\gamma^4}}         \right\} ,\quad &\gamma \leq \sqrt{8/3}  \\
\frac13 \left( 4 + \gamma^2 +\sqrt{16 + 2 \gamma^2 + \gamma^4} \right) ,\quad &\gamma \geq \sqrt{8/3} 
\end{dcases}
\eqe 
and
\eqb  \label{eqn-d-upper}
\ol d_\gamma := 
\begin{dcases}
\min\left\{    \frac13 \left( 4 + \gamma^2 +\sqrt{16 + 2 \gamma^2 + \gamma^4} \right)  ,  2 + \frac{\gamma^2}{2} + \sqrt 2 \gamma \right\} ,\quad &\gamma \leq \sqrt{8/3}  \\
\frac{ 12 - \sqrt 6 \gamma + 3 \sqrt{10} \gamma + 3 \gamma^2}{4 + \sqrt{15}}  ,\quad &\gamma \geq \sqrt{8/3} 
\end{dcases} .
\eqe 
\end{cor}

See Figure~\ref{fig-d-bound} for a plot of the bounds for Corollary~\ref{cor-d-bound}, the previous best known bounds from~\cite{dg-lqg-dim}, and the Watabiki prediction~\eqref{eqn-watabiki}. 
The new bounds are still consistent with Watabiki's prediction for $\gamma \in (0,2]$ (since $2/d_2  < 1/\sqrt 3$). 
The upper (resp.\ lower) bound from Corollary~\ref{cor-d-bound} is strictly better than previously known bounds in the case when $\gamma \geq \sqrt{8/3}$ (resp.\ $\gamma \in (0.4981,\sqrt{8/3})$). 
For $\gamma \geq \sqrt{8/3}$, the upper bound differs from Watabiki's prediction by at most $0.008$.

\begin{figure}[t!]
 \begin{center}
\includegraphics[scale=.55]{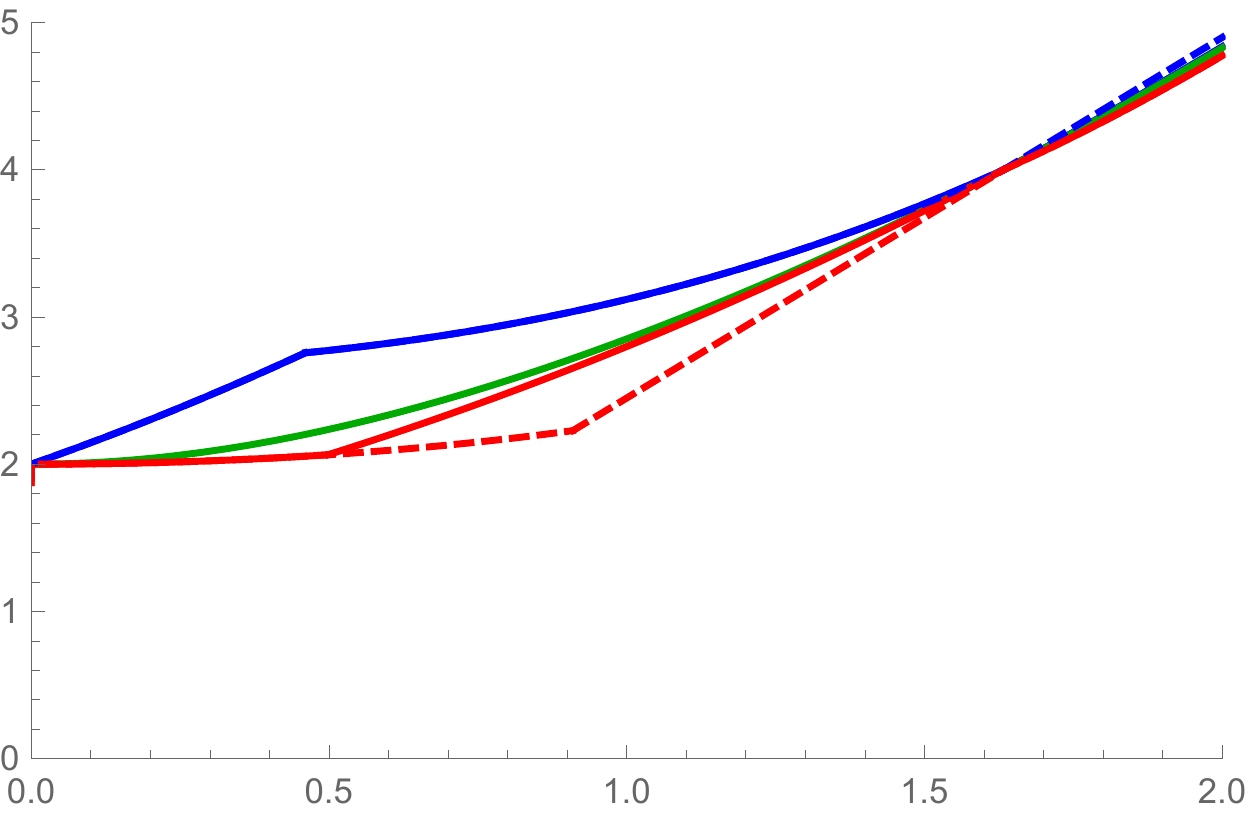} \hspace{15pt} \includegraphics[scale=.55]{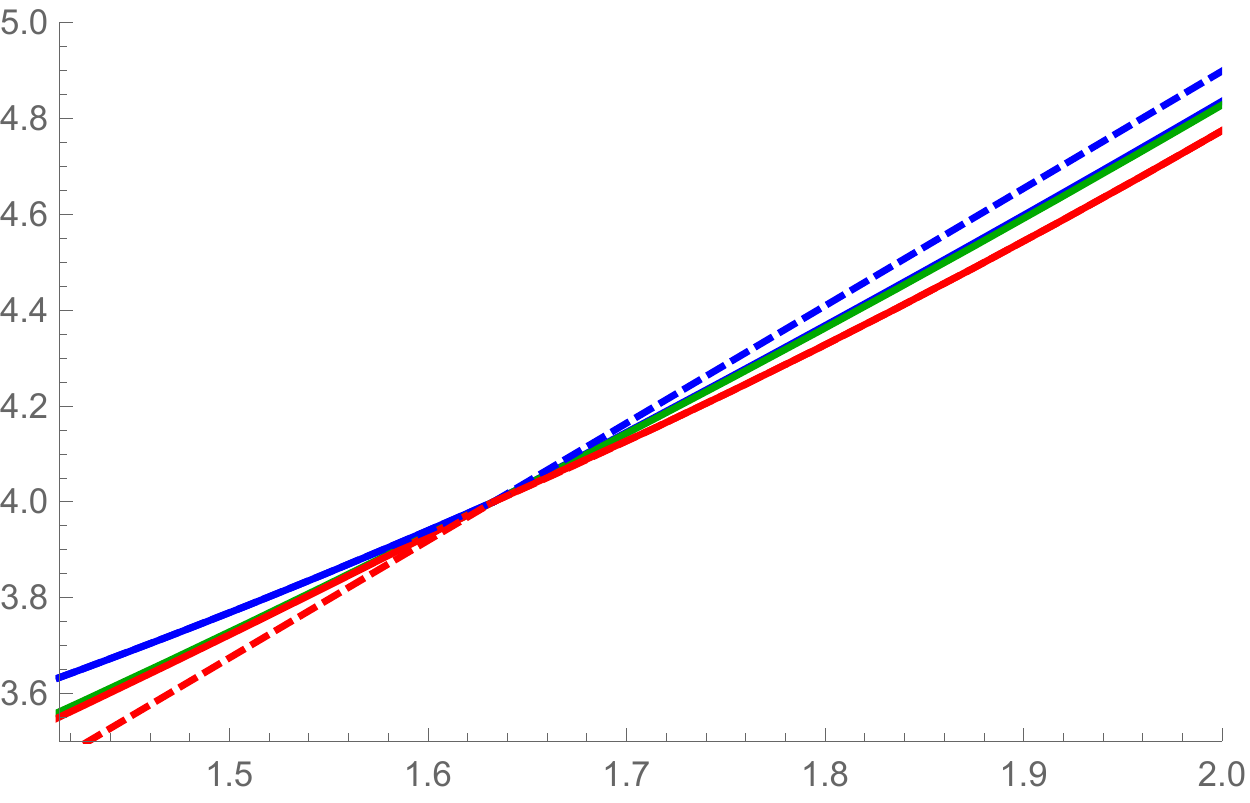}
\vspace{-0.01\textheight}
\caption{ \textbf{Left.} Graph of the lower bound $\ul d_\gamma$ (red) and the upper bound $\ol d_\gamma$ (blue) from Corollary~\ref{cor-d-bound} together with the Watabiki prediction $d_\gamma^{\op{Wat}}$ from~\eqref{eqn-watabiki} (green) and the previous best known bounds (dashed blue and dashed red). Note that the bounds $\ul d_\gamma \leq d_\gamma \leq \ol d_\gamma$ are consistent with the Watabiki prediction (so the only contradiction to Watabiki for $\gamma \in (0,2)$ is still~\cite{ding-goswami-watabiki}).  
\textbf{Right.} Graph of the same functions but restricted to the interval $[\sqrt 2 ,2]$.  
}\label{fig-d-bound}
\end{center}
\vspace{-1em}
\end{figure}

\begin{remark} \label{remark-ang}
The recent paper~\cite{ang-discrete-lfpp} shows that $\lambda(\xi) \geq 0$ for a slightly different definition of $\lambda(\xi)$, using distance between the inner and outer boundaries of an annulus, rather than distance across a circle. All of our arguments still work with this alternative definition of $\lambda(\xi)$ (the definitions are known to be equivalent when $\xi \leq 2/d_2$, and we expect that they are equivalent for all $\xi \geq 0$). 
The bound $\lambda(\xi) \geq 0$ improves on our lower bound for $\lambda(\xi)$ from Theorem~\ref{thm-lambda-bound} in the case when $\xi \in (0,0.2661\dots) \cup (1/\sqrt 2 , \infty)$.
By~\eqref{eqn-lambda-gamma}, for $\gamma \in (0,2)$, $\lambda(\xi ) \geq 0$ implies that $d_\gamma \geq 2+\gamma^2/2$, which improves on our lower bound for $d_\gamma$ from Corollary~\ref{cor-d-bound} in the case when $\gamma \in (0, 0.5765\dots)$. 
\end{remark}

\subsection*{Geodesic dimension}

In addition to the Hausdorff dimension/distance exponent, another natural quantity associated with the LQG metric is the Euclidean dimension of its geodesics. 
Theorem~\ref{thm-length-compare} also leads to a bound for this dimension. 
Let $P_h^{\xi,\ep}$ be the a.s.\ unique path in $\BB S^\ep$ connecting the left and right boundaries with minimal $D_h^{\xi,\ep}$-length.\footnote{The uniqueness of $P^{\xi,\ep}$ follows since there are only finitely many simple paths in $\BB S^\ep$ and a.s.\ no two such paths have the same $D_h^{\xi,\ep}$-length, which in turn is a consequence of the fact that a.s.\ $h_\ep(z)\not=h_\ep(w)$ for each distinct $z,w\in \BB S^\ep$.} We define the LFPP geodesic dimension
\eqb \label{eqn-geo-dim-def}
g(\xi) := \inf\left\{ \alpha  > 0 : \liminf_{\ep\rta 0} \BB P\left[ \# P_h^{\xi,\ep} \leq \ep^{-\alpha}   \right] = 1\right\} ,
\eqe
where here and in what follows we write $\# P$ for the number of vertices in a path $P$. 
Then $g(\xi)$ is a reasonable notion of the ``Euclidean dimension" of the path $P_h^{\xi,\ep}$ since $\# P$ is the number of Euclidean squares of side length $\ep$ needed to cover $P$. If (as expected) $D_h^{\xi,\ep}$ converges to a limiting metric as $\ep\rta 0$, then $g(\xi)$ should be the Euclidean Minkowski dimension of a typical geodesic for this limiting metric. 

In~\cite{ding-zhang-geodesic-dim} it is shown that $g(\xi) > 1$ whenever $\lambda(\xi) > 0$, which we know is the case for $\xi$ sufficiently small by~\cite{ding-goswami-watabiki} and for $\xi  > 0.266\dots$ by Theorem~\ref{thm-lambda-bound}. 
We also note that~\cite[Proposition 4.8]{mq-geodesics} shows that the Hausdorff dimension of geodesics for the continuum $\gamma$-LQG metric is strictly less than 2. 
Neither of these works prove a non-trivial explicit bound.
Here we prove the first non-trivial explicit bound for $g(\xi)$.  

\begin{cor}[Geodesic dimension upper bound] \label{cor-geodesic-dim-upper}
For each $\xi > 0$ and each $\zeta >0$, it holds with probability tending to 1 as $\ep\rta 0$ that each simple path $P$ in $\BB S^\ep$ with $D_h^{\xi,\ep}$-length $L_h^{\xi,\ep}(P) \leq \ep^{\lambda(\xi) - \zeta}$ satisfies
\eqb \label{eqn-geodesic-dim-upper}
\# P \leq \ep^{\lambda(\xi)   -  \xi \left( \sqrt{2 + 2 \lambda(\xi) + \xi^2}  -\xi  \right)  - 1  - \zeta } .
\eqe 
In particular, 
\eqb \label{eqn-g-upper}
g(\xi) \leq 1 - \lambda(\xi)   +   \xi \left( \sqrt{2 + 2 \lambda(\xi) + \xi^2}  -\xi  \right) .
\eqe
\end{cor}
\begin{proof}
Apply Theorem~\ref{thm-length-compare} with $\wt\xi =  0$ and note that $L_h^{0,\ep}(P) = \ep \# P$. 
\end{proof}

The upper bound~\eqref{eqn-g-upper} is a decreasing function of $\lambda(\xi)$ whenever $\lambda(\xi) \geq -1$, and we know that $\lambda(\xi) \geq -1/2$ by Lemma~\ref{lem-lambda-basic}. 
Plugging our lower bound for $\lambda(\xi)$ from Theorem~\ref{thm-lambda-bound} into~\eqref{eqn-g-upper} gives an upper bound for $g(\xi)$ in terms of $\xi$ which is non-trivial ($ < 2$) for $\xi  < \sqrt{5/2}$. We plot this bound in Figure~\ref{fig-lambda-bound}, right.
Recall from the discussion just after~\eqref{eqn-geo-dim-def} that $g(\xi)$ is expected to be the Euclidean dimension of geodesics with respect to the continuum limit of the metrics $D_h^{\xi,\ep}$. By setting $\xi = \gamma/d_\gamma$ for $\gamma \in (0,2)$ in~\eqref{eqn-g-upper} and using~\eqref{eqn-lambda-gamma}, we get the following heuristic bound: 
\eqb  \label{eqn-geodesic-gamma}
\left( \text{Euclidean dimension of $\gamma$-LQG geodesics}\right)\leq \frac{\gamma}{d_\gamma}\left( \frac{2}{\gamma} + \frac{\gamma}{2} - \frac{\gamma}{d_\gamma}    + \sqrt{2 + 2 \frac{\gamma}{d_\gamma} + \frac{\gamma^2}{d_\gamma^2} }    \right)  .
\eqe 
For $\gamma =\sqrt{8/3}$, in which case $d_\gamma = \sqrt{8/3}$, the right side of~\eqref{eqn-geodesic-gamma} is 
$\frac{1}{6} (4 + \sqrt{15}) \approx 1.31216$.
In~\cite{gp-kpz}, we prove~\eqref{eqn-geodesic-gamma} for geodesics of the continuum $\gamma$-LQG metric from~\cite{gm-uniqueness}.

\section{Proof of Theorem~\ref{thm-length-compare}}
\label{sec-proof}

The only estimate needed for the proof of Theorem~\ref{thm-length-compare} is the following lemma. 

\begin{lem} \label{lem-bad-circle-count}
For $\alpha >0$, 
\eqbn
\BB E\left[ \#\left\{ z \in \BB S^\ep : h_\ep(z) <  \alpha\log\ep  \right\}  \right] = O_\ep\left(  \ep^{ -(2-\alpha^2/2)  } \right) .
\eqen
\end{lem}
\begin{proof}
The calculations in~\cite[Section 3.1]{shef-kpz} show that for each vertex $z \in \BB S^\ep$, the circle average $ h_\ep(z)$ is centered Gaussian with variance $\log \ep^{-1} + O_\ep(1)$, where the $O_\ep(1)$ is uniform over all $z\in\BB S^\ep$.\footnote{The calculations in~\cite[Section 3.1]{shef-kpz} are carried out for the zero-boundary GFF on a proper subdomain of $\BB C$, but similar calculations work in the whole-plane case. Alternatively, the whole-plane case can be extracted from the zero-boundary case using either the Markov property of the whole-plane GFF~\cite[Proposition 2.8]{ig4}; or the fact that $h|_{(-1,2)^2}$ can be expressed as the limit (in the total variation sense) of $(\rng h^n - \rng h^n_1(0))|_{(-1,2)^2}$ as $n\rta\infty$, where for $n\in\BB N$, $\rng h^n$ is the zero-boundary GFF on the ball $B_n(0)$~\cite[Proposition 2.10]{ig4}.}
The lemma now follows by applying the Gaussian tail bound to each of these random variables, then summing over the $O_\ep (\ep^{-2})$ vertices of $\BB S^\ep$. 
\end{proof}

\begin{proof}[Proof of Theorem~\ref{thm-length-compare}]
Fix $\alpha > 0$ to be chosen later, in a manner depending only on $\xi$ and $\wt\xi$. 
Our strategy for bounding the $\wt{\xi}$-LFPP length of a path $P$ in terms of its $\xi$-LFPP length is to partition the set of points $z \in P$ according to whether $h_{\ep}(z) < \alpha \log \ep$ or $h_\ep(z) \geq \alpha\log\ep$.   The idea is that the contribution of points $z$ with $h_{\ep}(z) \geq \alpha \log \ep$ to the $\wt{\xi}$-LFPP length can be bounded in terms of their contribution to the $\xi$-LFPP length; on the other hand, we can crudely bound the set of $z \in P$ with $h_{\ep}(z) < \alpha \log \ep$ in terms of the total number of such points in $\BB S^{\ep}$.

Let us now proceed with the details. 
By Lemma~\ref{lem-bad-circle-count}, it holds with probability tending to 1 as $\ep\rta 0$ that
\eqb \label{eqn-use-bad-circle-count}  
\#\left\{ z \in \BB S^\ep : h_\ep(z) <  \alpha\log\ep \right\} \leq  \ep^{ -(2-\alpha^2/2)  - \zeta } .
\eqe 
Henceforth assume that~\eqref{eqn-use-bad-circle-count} holds. We will show that~\eqref{eqn-length-compare} holds.  

Let $P : \{0,1,\ldots,N\} \rta \BB S^\ep$ be a simple path in $\BB S^\ep$ with $L_h^{\xi,\ep}(P)$-length at most $\ep^{\lambda(\xi) -\zeta} $.  
We partition the points in the sum defining $L_h^{\wt\xi,\ep}(P)$ according to whether $h_\ep(P(j)) < \alpha\log\ep$ or $h_\ep(P(j)) \geq \alpha\log\ep$:
\allb \label{eqn-path-split}
L_h^{\wt\xi,\ep}(P) 
= \sum_{j=0}^N \ep e^{\wt\xi h_\ep(P(j) )} 
&= \sum_{j : h_\ep( P(j) )  <   \alpha\log\ep } \ep e^{\wt\xi h_\ep(P(j) )} 
 + \sum_{j : h_\ep( P(j) )  \geq \alpha\log\ep } \ep e^{\wt\xi h_\ep(P(j) )} \notag \\
&\leq  \ep^{1 + \alpha \wt\xi} \#\left\{ j : h_\ep(P(j) )   <   \alpha\log\ep   \right\} 
 +   \sum_{j : h_\ep(P(j) )  \geq \alpha\log\ep } \ep e^{\wt \xi h_\ep(P(j) )}  .
\alle

Since $P$ is a simple path, the bound~\eqref{eqn-use-bad-circle-count} shows that the first term on the right in~\eqref{eqn-path-split} is at most $\ep^{  \wt\xi \alpha   + \alpha^2/2 - 1   - \zeta }$.  
As for the second term, since $\wt\xi\leq \xi$, if $h_\ep(P(j) ) \geq \alpha\log\ep$, then $e^{\wt \xi h_\ep(P(j) )} \leq \ep^{- (\xi-\wt\xi)\alpha } e^{ \xi h_\ep(P(j) )}$. 
Plugging these two estimates into~\eqref{eqn-path-split} shows that
\eqb \label{eqn-split-exponents}
L_h^{\wt\xi,\ep}(P)  
\leq \ep^{  \wt\xi  \alpha  + \alpha^2/2 - 1  + o_\ep(1)} +  \ep^{-(\xi-\wt\xi) \alpha } L_h^{\xi,\ep}(P)  
\leq \ep^{ \wt\xi \alpha + \alpha^2/2 - 1  - \zeta} + \ep^{\lambda(\xi) - (\xi-\wt\xi) \alpha  -\zeta} .
\eqe  
We now choose $\alpha >0$ so that the two powers on $\ep$ on the right in~\eqref{eqn-split-exponents} are equal, i.e., 
\eqb 
\alpha =   \sqrt{2 + 2 \lambda(\xi) + \xi^2}  -\xi .
\eqe 
Plugging this into~\eqref{eqn-split-exponents} gives~\eqref{eqn-length-compare}.
\end{proof}

\section{Relating $\xi$ to the central charge} 
\label{sec-cc}

The exponent $\lambda(\xi)$ of~\eqref{eqn-lambda-def} gives rise to a notion of ``central charge" for LFPP with exponent $\xi$, as we will now explain.
Following~\cite{shef-kpz}, one can define a \emph{$\gamma$-Liouville quantum gravity (LQG)} surface for $\gamma \in (0,2]$ to be an equivalence class of pairs $(U,h)$ where $U\subset\BB C$ is a simply connected domain and $h$ is a random distribution on $U$ (always taken to be a realization of some variant of the GFF on $U$), with two pairs $(U,h )$ and $(\wt U  ,\wt h)$ considered to be equivalent if there is a conformal map $\phi : \wt U \rta U$ such that
\eqb \label{eqn-lqg-coord}
\wt h = h\circ \phi  + Q\log |\phi'| \quad \text{for} \quad Q = \frac{2}{\gamma} +\frac{\gamma}{2} \geq 2 .
\eqe
We think of equivalent pairs as different parametrizations of the same surface. 
Objects associated with LQG are invariant under coordinate changes of the form~\eqref{eqn-lqg-coord}. 
This is proven for the measure in~\cite[Proposition 2.1]{shef-kpz} and the metric in~\cite[Theorem 1.1]{gm-coord-change}.

The parameter $Q$ in~\eqref{eqn-lqg-coord} is called the \emph{background charge}.
It is related to the so-called \emph{central charge} by $\cc = 25 - 6Q^2$.
We have $\cc \in (-\infty,1]$ for $\gamma \in (0,2]$.  
In the physics literature, the parameter $\cc$, rather than the parameter $\gamma$, is often viewed as the more natural one. 
The above definition of an LQG surface makes sense for any value of $Q  >0$ (not just $Q \geq 2$) and hence for any central charge $\cc \in (-\infty, 25)$. 
See \cite{ghpr-central-charge} for further discussion of LQG with $\cc  \in (1,25)$.

It is not hard to see (see~\cite[Proposition 2.3]{dg-lqg-dim}) that if $D_h^{\xi,\ep}$ has a scaling limit, then at least for complex affine maps $\phi$ the limiting metric must be invariant under coordinate changes of the form~\eqref{eqn-lqg-coord} for $Q = (1-\lambda(\xi))/\xi$. 
This leads us to define the background charge and central charge, respectively, for LFPP with parameter $\xi$ by
\eqb \label{eqn-Q-c}
Q(\xi) := (1-\lambda(\xi)) / \xi \quad \text{and} \quad \cc(\xi) := 25-6Q(\xi)^2 .
\eqe 
It is shown in~\cite[Theorem 1.5]{dg-lqg-dim} that for $\gamma\in (0,2)$, one has $Q(\gamma/d_\gamma) = 2/\gamma+\gamma/2$, as expected. 
Since $\lambda(\xi) \leq 1$, one always has $Q(\xi) \geq 0$. 

For $\xi = 1/\sqrt 3$, the extended Watabiki prediction~\eqref{eqn-watabiki-extend} gives $\lambda(\xi) = 1/3$ and hence $Q(\xi) = \sqrt{4/3}$ and $\cc(\xi) = 17$.
Similarly, under~\eqref{eqn-watabiki-extend}, $\xi = \sqrt{5/2} - \sqrt{2/3}$ corresponds to $\cc(\xi) = 21.741\dots$. 
Combined with Corollary~\ref{cor-watabiki-contradict} and Lemma~\ref{lem-Q-mono} just below, this means that the extended Watabiki prediction~\eqref{eqn-watabiki-extend} is false for a dense subset of central charge values in $(17,25)$ (resp.\ for all $\cc  \in ( 21.741\dots , 25)$). 

\begin{lem} \label{lem-Q-mono}
The background charge $Q(\xi)$ is strictly decreasing on $\left(0 , 0.7  \right)$, non-increasing on $[0.7,\infty)$, and satisfies $\lim_{\xi\rta\infty} Q(\xi) = 0$. 
\end{lem}
\begin{proof}
Since $\lambda(\xi) \in [-1/2,1]$ (Lemma~\ref{lem-lambda-basic}), it is obvious that $\liminf_{\xi\rta\infty} Q(\xi) = 0$. 

Since $Q(\xi)$ is a Lipschitz continuous function of $\xi$ (Lemma~\ref{lem-lambda-basic}), it is absolutely continuous and so is differentiable Lebesgue-a.e.
Hence to show that $Q(\xi)$ is strictly decreasing on $(0,0.7)$ it suffices to show that its derivative is strictly negative there. 
By~\eqref{eqn-dg-mono}, $\lambda'(\xi) \geq -\xi$ and hence
\eqbn
Q'(\xi)  = \frac{1}{\xi^2} \left( -\lambda'(\xi) \xi - 1 + \lambda(\xi) \right) \leq \frac{1}{\xi^2} \left( - \xi^2 - 1 + \lambda(\xi) \right) .
\eqen
Plugging in our upper bound for $\lambda(\xi)$ from Theorem~\ref{thm-lambda-bound} shows that this is negative for $\xi < \sqrt{2 - \frac{1}{6} \sqrt{113 - 8 \sqrt{15}}} \approx 0.70044$. 

Finally, we show $Q(\xi)$ is always non-increasing. Let $ 0 \leq \wt\xi \leq \xi$.
For $\ep > 0$, let $\wt P =  P_h^{\wt\xi,\ep}$ be the minimal-$D_h^{\wt\xi,\ep}$-length path between the left and right boundaries of $\BB S^\ep$. 
Since $x\mapsto x^{\wt\xi/\xi}$ is subadditive,
\eqbn
\left[ L_h^{\xi , \ep}(\wt P) \right]^{\wt\xi/\xi}  =   \ep^{\wt\xi/\xi} \left( \sum_{j=0}^N   e^{ \xi h_\ep(\wt P(j) )} \right)^{\wt\xi/\xi} \leq  \ep^{\wt\xi/\xi - 1} L_h^{\wt\xi , \ep}(\wt P)  .
\eqen
By the definition~\eqref{eqn-lambda-def} of $\lambda$, it holds with probability tending to 1 as $\ep\rta 0$ that 
\eqbn
\left[ D_h^{\xi,\ep}\left(\bdy_{\op{L}} \BB S , \bdy_{\op{R}} \BB S \right) \right]^{\wt\xi/\xi} \leq \left[ L_h^{\xi,\ep}(\wt P) \right]^{\wt\xi/\xi} \leq   \ep^{\wt\xi/\xi  +  \lambda(\wt\xi)   - 1  +o_\ep(1)}  .
\eqen
Consequently, $\lambda(\xi) \wt\xi/\xi \geq \wt\xi / \xi + \lambda(\wt\xi)   - 1$. Re-arranging gives 
$(1 -  \lambda(\xi))/\xi \leq (1-\lambda(\wt\xi))/\wt \xi  $.
\end{proof}

\begin{remark} \label{remark-square-subdivision}
The paper \cite{ghpr-central-charge} introduces another natural discretization of LQG which works for all $\cc  < 25$ (equivalently, $Q > 0$), based on a dyadic tiling $\mcl S_h^{Q,\ep}$ of the plane consisting of squares which all have ``LQG size $\ep$" with respect to $h$. We expect that this model is related to LFPP as follows: if $ \xi(Q)  > 0$ is chosen so that the graph distance in $\mcl S_h^{Q,\ep}$ between $\bdy_{\op{L}} \BB S$ and $\bdy_{\op{R}}\BB S$ grows like $\ep^{-\xi(Q)}$ as $\ep\rta 0$, then $\lambda(\xi(Q)) = 1 - \xi(Q) Q$; i.e., $Q(\xi(Q)) = Q$. 
This relation for $Q > 2$ and $\xi \in (0,2/d_2)$ follows from~\cite[Theorem 1.5]{dg-lqg-dim}. We expect that the proof of that theorem could be adapted to treat the case when $Q   < 2$ and $\xi  >2/d_2$ as well. 
\end{remark}

\begin{remark} \label{remark-Q-guess}
We expect that the function $\xi\mapsto Q(\xi)$ is injective, at least up until some ``critical point" $\xi_*$ at which it becomes constant. 
By Remark~\ref{remark-square-subdivision}, each value of $Q\in (0,2)$ should correspond to some value of $\xi > 0$, so if such a critical point $\xi_*$ exists then $\xi_*$ should be the smallest $\xi   > 0$ for which $\lambda(\xi) = 1$, equivalently $Q(\xi) = 0$. 
However, we do not know whether there exists $\xi > 0$ for which $Q(\xi) = 0$, so it could be the case that $\xi\mapsto Q(\xi)$ is strictly decreasing on all of $(0,\infty)$.  
Note that both the extended Watabiki prediction~\eqref{eqn-watabiki} and the extended alternative guess~\eqref{eqn-dg-guess} would say that $Q(\xi) = 0$ for some finite $\xi$, but we know that the former is not correct (Corollary~\ref{cor-watabiki-contradict}) and, as noted above, we have no theoretical justification for the latter.
\end{remark}

\bibliography{cibiblong,cibib}
\bibliographystyle{hmralphaabbrv}

\end{document}